\documentclass[11pt,a4paper]{amsart}
\usepackage{amssymb,amsmath,amsthm}

\usepackage[utf8]{inputenc}
\usepackage[T1]{fontenc}

\usepackage[pdftex]{graphicx}

\usepackage[colorlinks=true, pdfstartview=FitV, linkcolor=blue, citecolor=blue, urlcolor=blue,pagebackref=false]{hyperref}

\usepackage{esint}
\usepackage{mathrsfs}

\usepackage{times,latexsym,color}

\usepackage{comment}
\usepackage{enumitem}

\newcommand{\BOX}{\ensuremath\Box}

\newtheorem{theorem}{Theorem}
\newtheorem*{theorem*}{Theorem}
\newtheorem{proposition}[theorem]{Proposition}

\theoremstyle{remark}
\newtheorem{remark}[theorem]{Remark}

\theoremstyle{definition}

\DeclareMathOperator{\Riesz}{R}

\newcommand{\N}{\mathbb{N}}

\newcommand{\R}{\mathbb{R}}

\newcommand{\ep}{\varepsilon}

\definecolor{darkgreen}{rgb}{0,0.5,0}
\definecolor{darkblue}{rgb}{0,0,0.7}
\definecolor{darkred}{rgb}{0.9,0.1,0.1}
\definecolor{lightblue}{rgb}{0,0.51,1}

\begin{document}

\title[Mild criticality breaking for the Navier-Stokes equations]{Mild criticality breaking for the Navier-Stokes equations}

\author[T. Barker]{Tobias Barker}
\address[T. Barker]{Mathematics Institute, University of Warwick, UK}
\email{tobiasbarker5@gmail.com}

\author[C. Prange]{Christophe Prange}
\address[C. Prange]{Cergy Paris Universit\'e, Laboratoire de Math\'ematiques AGM, UMR CNRS 8088, France}\email{christophe.prange@cyu.fr}

\maketitle

\noindent {\bf Abstract} In this short paper we prove the global regularity of solutions to the Navier-Stokes equations under the assumption that slightly supercritical quantities are bounded. As a consequence, we prove that if a solution $u$ to the Navier-Stokes equations blows-up, then certain slightly supercritical Orlicz norms must become unbounded. This partially answers a conjecture recently made by Terence Tao. The proof relies on quantitative regularity estimates at the critical level and transfer of subcritical information on the initial data to arbitrarily large times. This method is inspired by a recent paper of Aynur Bulut, where similar results are proved for energy supercritical nonlinear Schr\"odinger equations. 
\vspace{0.3cm}


\vspace{0.3cm}


\section{Introduction}

This paper is concerned with regularity criteria for weak Leray-Hopf solutions to the three-dimensional Navier-Stokes equations
\begin{equation}\label{e.nse}
\partial_tu-\Delta u+u\cdot\nabla u+\nabla p=0,\qquad \nabla\cdot u=0,\qquad u(\cdot,0)=u_{0}(x),\qquad\textrm{in}\,\,\,\mathbb{R}^3\times (0,\infty).
\end{equation}
A weak Leray-Hopf solution $u$ is a distributional solution of the Navier-Stokes equations, is continuous in time with respect to the weak $L_{2}$ topology and satisfies the energy inequality
\begin{equation}\label{weakLeraydef}
E(u)(t):=\|u(\cdot,t)\|^2_{L^{2}(\mathbb{R}^3)}+2\int\limits_{0}^{t}\int\limits_{\mathbb{R}^3} |\nabla u(x,s)|^2 dxds\leq \|u_{0}\|^2_{L^{2}(\mathbb{R}^3)}.
\end{equation}
It is known that the Navier-Stokes equations are invariant with respect to the scaling symmetry
\begin{equation}\label{NSrescale}
(u_{\lambda}(x,t),p_{\lambda}(x,t))=(\lambda u(\lambda x, \lambda^2 t), \lambda^2 p(\lambda x, \lambda^2 t)),\,\,\,u_{0\lambda}(x)=\lambda u_{0}(\lambda x). \end{equation}
 Most of the known conditions ensuring regularity of the Navier-Stokes equations are formulated in terms of \textit{scale-invariant} quantities such as $u\in L^{5}_{x,t}$ \cite{ladyzhenskaya1967uniqueness} or $u\in L^{\infty}_{t}L^{3}_{x}$ \cite{ESS2003}.
The main obstacle to obtaining regularity of three-dimensional weak Leray-Hopf solutions of the Navier-Stokes equations without any additional assumptions is that the energy associated the solution is not scale-invariant. In particular,
$$E(u_{\lambda})(t)=\lambda^{-1}E(u)(\lambda^2 t).$$
Such quantities  that  scale to a negative power of $\lambda$ under the Navier-Stokes rescaling \eqref{NSrescale} are known as \textit{supercritical}.

Recently, under the assumption of  certain finite `slightly supercritical' quantities (that scale to a logarithm of $\lambda$ under the Navier-Stokes rescaling), regularity was shown \cite{pan2016regularity} and \cite{Tao19}. The aim of this paper is to provide another new example of a slightly supercritical regularity criteria for the Navier-Stokes equations. Our main theorem is the following. Note that throughout this paper, we use the notation of `suitable weak Leray-Hopf solutions'. In addition to being weak Leray-Hopf solutions, such solutions satisfy the local energy inequality in $\mathbb{R}^3\times (0,\infty)$ as defined in \cite[Definition 6.1]{gregory2014lecture}.

\begin{theorem}\label{theo.main}
For all $M\in[1,\infty)$ and $E\in[1,\infty)$ sufficiently large,\footnote{\label{foot.one}This means that there exists a universal constant $N_{univ}\in[1,\infty)$ such that for all $M\geq N_{univ}$ and $E\geq N_{univ}$ we have the result.} 
there exists $\delta(M,E)\in(0,\frac12]$ such that the following holds. Let $u$ be a suitable weak Leray-Hopf  solution to the Navier-Stokes equations \eqref{e.nse} on $\R^3\times(0,\infty)$ with initial data $u_0\in L^2(\R^3)\cap L^4(\R^3)$.\\
Assume that 
\begin{equation*}
\|u_0\|_{L^2},\ \|u_0\|_{L^4}\leq M,
\end{equation*}
and that
\begin{equation}\label{slightlysupercrithypo}
\|u\|_{L^\infty(0,\infty;L^{3-\delta(M,E)}(\R^3))}\leq E.
\end{equation}
Then, the above assumptions imply that $u$ is smooth on $\R^3\times (0,\infty)$. 
Moreover, there is an explicit formula for $\delta(M,E)$, see \eqref{e.defdelta} below, and $\delta(M,E)\rightarrow 0$ when $E\rightarrow\infty$ or $M\rightarrow\infty$.
\end{theorem}
Our method and result is directly inspired by the recent result of Bulut \cite{Bulut20} for the nonlinear supercritical Schr\"{o}dinger equation. In particular, Theorem 1 hinges on quantitative bounds for a Navier-Stokes solution belonging to the critical space $L^{\infty}_{t}L^{3}_{x}$, which were established by Tao in \cite{Tao19}; see also the subsequent paper by the authors \cite{BP20}. For other partial differential equations, it is often the case that a refined understanding of critical regimes can be used to prove `slightly supercritical' results. Such slightly supercritical results occur for the nonlinear wave equation \cite{tao2007global,CH19-wave}, the hyperdissipative Navier-Stokes equations \cite{tao2010global,barbato2015global,colombo2020generalized,CH19-NShyp}, the supercritical SQG equation \cite{dabkowski2012global,CZV16} and the fractional Burgers equation \cite{dabkowski2014global} to name a few.  
In these works the slight supercriticality is obtained by varying the power of the nonlinearity as in the first paper, or the strength of the fractional dissipation as in the last three. 

We call the result of Theorem \ref{theo.main} a \emph{mild breaking of the criticality}, or a \emph{mild supercritical regularity criteria}. Indeed, the supercritical space $L^{\infty}_tL^{3-\delta(M,E)}$ in which we break the scaling depends on the size $E$ of the solution in this supercritical space via $\delta(M,E)$. In other words this can be considered as a non effective regularity criteria, hence the term \emph{mild}. Moreover, given a solution $u$, assume that you knew all the $L^\infty_tL^{3-\delta}_x$ norms for $\delta\rightarrow 0$. Then the question whether Theorem \ref{theo.main} applies to $u$ or not becomes a question about how fast 
\begin{equation*}
\|u\|_{L^\infty(0,\infty;L^{3-\delta}(\R^3))}
\end{equation*} 
grows when $\delta\rightarrow 0$. 
Moreover, in view of (Step 3) below in Section \ref{sec.mildquant}, the fact that $\delta(M,E)$ depends on $E$ seems unavoidable; see in particular the definition of $\delta(M,E)$ in \eqref{e.defdelta}. As far as we know the only space in which regularity is known for arbitrarily large norm is $L^\infty_tL^3_x$. Let us also remark that the condition $u_0\in L^4(\R^3)$ can be replaced by any subcritical condition $u_0\in L^{3+}(\R^3)$.

The mild criticality breaking of Theorem \ref{theo.main} is also in contrast with \emph{strong} criticality breaking, which is genuinely supercritical and not a consequence of the critical theory. Only very few results enter that category for the Navier-Stokes equations. In this vein, let us mention the regularity criteria of Pan \cite{pan2016regularity} for slightly supercritical axisymmetric solutions to Navier-Stokes.

Let us now state a consequence of Theorem \ref{theo.main}.
\begin{theorem}\label{cor.orliczblowup}
There exists a universal constant $\theta\in(0,1)$ such that the following holds. 

Let $u$ be a weak Leray-Hopf 
solution to the Navier-Stokes equations \eqref{e.nse} on $\R^3\times(0,\infty)$ with initial data $u_0\in L^2(\R^3)\cap L^4(\R^3)$. Assume that $u$ first blows-up at $T^*>0$, namely
\begin{equation*}
u\in L^{\infty}_{loc}(0,T^*; L^{\infty}(\mathbb{R}^3))\,\,\,\textrm{and}\,\,\,u\notin L^{\infty}((\tfrac{1}{2}T^*, T^*);L^{\infty}(\mathbb{R}^3)).
\end{equation*}
Then the above assumptions imply that 
\begin{equation}\label{e.orliczblowup}
\lim\sup_{t\uparrow T^*}\int\limits_{\R^3}\frac{|u(x,t)|^3}{\Big(\log\log\log\big((\log(e^{e^{3e^{e}}}+|u(x,t)|))^\frac13\big)\Big)^\theta}dx=\infty.
\end{equation}
\end{theorem}
Recently, in \cite{Tao19} (Remark 1.6), Tao conjectured that if a solution first loses smoothness at time $T^*>0$, then the Orlicz norm $\|u(t)\|_{L^{3}(\log\log\log L)^{-c}(\mathbb{R}^3)} $ must blow-up as $t$ tends to $T^*$. Theorem \ref{cor.orliczblowup} provides a positive answer to Tao's conjecture, albeit with an extra logarithm in the denominator. The proof of Theorem \ref{cor.orliczblowup} is a consequence of Theorem \ref{theo.main}. Its proof in Section \ref{sec.orlicz} relies on a careful tuning of the parameters.

Previously, it was shown in \cite{CV07} 
that if $u$ is a weak Leray-Hopf solution satisfying $$\int\limits_{0}^{\infty}\int\limits_{\mathbb{R}^3}\frac{|u|^5}{\log(1+|u|)}dxdt<\infty $$ then $u$ is smooth on $\mathbb{R}^3\times (0,\infty)$. Subsequent improvements were obtained in \cite{LZ13} 
and \cite{BV11}. Let us mention that the techniques used in these papers cannot be used to treat the \textit{endpoint} case we consider in Theorem \ref{cor.orliczblowup}.

\noindent{\bf Alternative (qualitative) proof of Theorem \ref{theo.main}.} A proof of Theorem \ref{theo.main} relying on quantitative arguments in given in Section \ref{sec.mildquant}. In order to clarify the role played by the hypothesis in Theorem \ref{theo.main}, we find it instructive to finish this introduction by sketching a \emph{qualitative argument} that yields Theorem \ref{theo.main}. Nevertheless, this argument does not give an explicit quantitative bound for $\delta(M,E)$, which is produced by the quantitative argument given in Section \ref{sec.mildquant} below. The argument is by contradiction and uses the persistence of singularities for critically bounded solutions \cite{rusin2011minimal}. Such a strategy is known to enable to non explicitly quantify regularity criteria, such as the celebrated Escauriaza, Seregin and \v Sver\'ak \cite{ESS2003} result; see \cite{rusin2011minimal} as well as \cite{albritton2018local} 
and \cite[Introduction]{BP20}.

We argue by contradiction. There exists a sequence of suitable weak Leray-Hopf solutions $u^{(k)}$ on $\R^3\times (0,\infty)$, initial data $u_0^{(k)}$ and $\delta^{(k)}\rightarrow 0$ such that
\begin{align*}
&\|u_0^{(k)}\|_{L^2},\quad\|u_0^{(k)}\|_{L^4}\leq M,\\
&\|u^{(k)}\|_{L^\infty_t(0,\infty;L^{3-\delta^{(k)}}(\R^3))}\leq E,\\
&u^{(k)}\quad\mbox{has a singular point}\quad (x^{(k)},t^{(k)})\in\R^3\times (0,\infty).
\end{align*}
Note that we say that $(x,t)$ is a singular point of  a Navier-Stokes solution $u$, if for all $r>0$ sufficiently small $u\notin L^{\infty}(B(x,r)\times (t-r^2,t))$.

\noindent{\underline{First step: trapping the singularity.}} By the theory of mild solutions with subcritical $L^4$ data \cite{Giga86} on the one hand, and eventual regularity \cite[Theorem 1.2]{BT19} on the other hand, there exists a universal constant $C\in(0,\infty)$ such that 
\begin{equation*}
\frac1{CM^{8}}\leq t^{(k)}\leq CM^4\quad\mbox{for all}\quad k\in\N.
\end{equation*}
By translation in space we can assume that $x^{(k)}=0$.

\noindent{\underline{Second step: persistence of singularities.}} Up to a subsequence $(0,t^{(k)})\rightarrow \bar t$ with $C^{-1}M^{-8}\leq\bar t\leq CM^4$. Moreover, the global energy of $u^{(k)}$ is uniformly bounded in $k$. By \cite{Ser12}, $u^{(k)}$ converges strongly in $L^3_{loc,t,x}$ (up to a subsequence) to a suitable weak Leray-Hopf solution $\bar v$ belonging to the energy space. By persistence of singularities \cite[Proposition 2.2]{rusin2011minimal}, $\bar v$ has a singularity at $(0,\bar t)$.

\noindent{\underline{Third step: $\bar v\in L^\infty_tL^3_x$.}} Since $\bar v\in L^\infty(0,\infty;L^2(\R^3))$ and $\nabla\bar v\in L^2(0,\infty;L^2(\R^3))$, we have $\bar v\in L^4(0,\infty;L^3(\R^3))$. Hence $\bar v(\cdot,t)\in L^3(\R^3)$ for almost every $t\in(0,\infty)$. Now for any fixed $\delta\in(0,\frac12)$, $0<\delta^{(k)}\leq\delta$ for $k$ large enough, and hence we can interpolate the $L^\infty(0,\infty;L^{3-\delta}(\R^3)))$ norm between the $L^\infty(0,\infty;L^2(\R^3))$ and $L^\infty(0,\infty;L^{3-\delta^{(k)}}(\R^3))$ norms. This yields
\begin{equation*}
\|u^{(k)}\|_{L^\infty(0,\infty;L^{3-\delta}(\R^3))}\leq  E^\frac{\frac{1}{3-\delta}-\frac12}{\frac{1}{3-\delta^{(k)}}-\frac12}M^{\frac{\frac{1}{3-\delta^{(k)}}-\frac{1}{3-\delta}}{\frac{1}{3-\delta^{(k)}}-\frac12}}\leq \max(E,M).
\end{equation*}
Now by Fatou's lemma, for almost every $t\in(0,\infty)$,
\begin{equation*}
\int\limits_{\R^3}|\bar v(x,t)|^3\, dx\leq\liminf_{\delta\rightarrow 0}\int\limits_{\R^3}|\bar v(x,t)|^{3-\delta}\, dx\leq\max(E,M).
\end{equation*}
Hence, $\bar v\in L^\infty(0,\infty;L^3(\R^3))$. From \cite{ESS2003}, this implies that $\bar{v}$ is smooth in $\mathbb{R}^3\times (0,\infty)$, which contradicts the fact that $\bar v$ has a singularity (see the second step above).

\section{Quantitative proof of regularity under boundedness of mild supercritical quantities}
\label{sec.mildquant}

The proof of Theorem \ref{theo.main} relies on three main steps:
\begin{enumerate}[label=(Step \arabic*)]
\item quantitative control of the critical $L^5_{t,x}$ norm, see Proposion \ref{prop.step1} below;
\item $L^p$ energy estimates for critically bounded solutions to the Navier-Stokes equations, see Proposition \ref{prop.step2} below;
\item transfer of subcritical information from the initial data to large time via slicing in time.
\end{enumerate}

\noindent{\bf Step 1: quantitative control of the critical $L^5_{t,x}$ norm.}

\begin{proposition}\label{prop.step1} 
For $M$ and $\bar E$ sufficiently large (see Footnote \ref{foot.one}), there exists $C(M)\in(0,\infty)$ such that the following holds. Let $T\in(0,\infty)$. Let $u$ be a suitable weak Leray-Hopf 
to the Navier-Stokes equations \eqref{e.nse} on $\R^3\times(0,T)$ with initial data $u_0\in L^2(\R^3)\cap L^4(\R^3)$.\\
Assume that 
\begin{equation*}
\|u_0\|_{L^2},\ \|u_0\|_{L^4}\leq M,
\end{equation*}
and that
\begin{equation*}
\|u\|_{L^\infty(0,T;L^{3}(\R^3))}\leq \bar E.
\end{equation*}
Then, we have the following quantitative estimate
\begin{equation}\label{e.estL5tx}
\|u\|_{L^5(0,T;L^5(\R^3))}\leq C(M)\exp\exp\exp\big(C_{univ}{\bar E}^c\big).
\end{equation}
Here, $c\in[1,\infty)$ is a universal constant and $C(M)$ can be taken to be $C_{univ}(\log(M))^\frac15$.
\end{proposition}

The point in this proposition is that we control quantitatively the critical $L^5_{t,x}$ norm. This enables, see (Step 3) below, to slice the time interval into disjoint epochs where the $L^5_{t,x}$ norm is small. Moreover, we remark that the constant in \eqref{e.estL5tx} is independent of time. 
Notice that by \cite{ESS2003} a suitable weak Leray-Hopf solution to \eqref{e.nse} for which 
$$\|u\|_{L^\infty(0,T;L^{3}(\R^3))}<\infty$$
 is a classical solution on $\R^3\times(0,T]$. Hence the quantitative regularity result of Tao \cite{Tao19} can be used off-the-shelf.

\begin{proof}[Proof of Proposition \ref{prop.step1}]
We combine: (i) short-time estimates coming for the mild solution theory \cite{Giga86}, with (ii) eventual regularity for suitable weak Leray-Hopf solutions in three dimensions \cite[Theorem 1.2]{BT19} and (iii) quantitative regularity estimates away from initial time. First, since the initial data $u_0\in L^4(\R^3)$, there exists a universal constant \begin{align}
\begin{split}\label{e.estshorttime}
\|u\|_{L^5(0,c_4\|u_0\|_{L^4}^{-8};L^5(\R^3))}\leq\ & c_4^{1/8}\|u_0\|_{L^4}^{-1}\|u\|_{L^\frac{40}3(0,c_4\|u_0\|_{L^4}^{-8};L^5(\R^3))}\\
\leq\ & C_{univ}
\end{split}
\end{align}
Second, by eventual regularity for suitable weak Leray-Hopf solutions \cite[Theorem 1.2]{BT19} we have
\begin{equation}\label{eventualreg}
\|u\|_{L^{\infty}_t(C_{univ}M^4,\infty; L^{\infty}_{x}(\mathbb{R}^3))}\leq C_{univ}M^{-2}.
\end{equation}
Moreover, interpolating \eqref{eventualreg} with  
$$\|u\|_{L^\frac{10}3(\R^3\times(0,T))}\leq C_{univ}\|u_0\|_{L^2},$$
coming from the finiteness of the global energy, we obtain
\begin{equation}\label{e.L5largetime}
\|u\|_{L^5(\R^3\times(C_{univ}M^4,\infty))}\leq (C_{univ}\|u_0\|_{L^2})^\frac23(C_{univ}M^{-2})^{\frac{1}{3}}\leq C_{univ}.
\end{equation}
Third, we consider the case of intermediate times $S\in(c_4\|u_0\|_{L^4}^{-8},C_{univ}M^4)$. By the quantitative regularity from \cite[Theorem 1.2]{Tao19} 
under the boundedness of the critical $L^\infty(0,S;L^3(\R^3))$ norm, we obtain for all $t\in(0,S)$, 
\begin{equation}\label{e.tao}
\|u(\cdot,t)\|_{L^{\infty}(\mathbb{R}^3)}\leq t^{-\frac12}\exp\exp\exp\big(C_{univ}\bar E^c\big)
\end{equation}
where $c\in(0,\infty)$ is a universal constant. Interpolating with $\|u(\cdot,t)\|_{L^{3}(\mathbb{R}^3)}$ yields that, for $\|u\|_{L^\infty(0,S;L^3(\R^3))}$ sufficiently large, we have
\begin{equation}\label{L5log}
\|u(\cdot,t)\|_{L^{5}(\mathbb{R}^3)}\leq t^{-\frac15}\exp\exp\exp\big(C_{univ}\bar E^c\big)
\end{equation}
Thus integrating over $(c_4\|u_0\|_{L^4}^{-8},S)$ and using that $S\leq C_{univ}M^4$ gives
\begin{equation}\label{L5intermediateest}
\|u\|_{L^{5}(c_4\|u_0\|_{L^4}^{-8},S; L^{5}(\mathbb{R}^3))}^5\leq C_{univ}\log(M)\exp\exp\exp\big(C_{univ}\bar E^c\big)
\end{equation} 
Finally, combining the short-time estimate \eqref{e.estshorttime} with the large-time estimate \eqref{e.L5largetime} and the intermediate-time estimate \eqref{L5intermediateest} yields the desired conclusion. \endproof
\end{proof}

\noindent{\bf Step 2: $L^p$ energy estimates for critically bounded solutions.} For $t_1\in\R$, $S\in(0,\infty)$ and $p\in [3,\infty)$, we define the $L^p$ energy as follows
\begin{multline}\label{e.Lpener}
\mathcal E_{p,t_1}(S):=\sup_{t_1\leq s\leq t_1+S}\int\limits_{\R^3}|u|^p(x,s)\, dx\\
+p\int\limits_{t_1}^{t_1+S}\int\limits_{\R^3}|\nabla u|^2|u|^{p-2}\, dxds+\frac{4(p-2)}{p}\int\limits_{t_1}^{t_1+S}\int\limits_{\R^3}|\nabla|u|^\frac p2|^2\, dxds.
\end{multline}
The following proposition provides estimates for this $L^p$ energy independently of $T$ under the assumption that $u$ is critically bounded and sufficiently smooth.

\begin{proposition}\label{prop.step2}
Let $t_1\in\R$, $S\in(0,\infty)$ and $p\in [4,\infty)$. There exists $C_p\in(0,\infty)$ such that the following holds. Let $u$ be a smooth bounded weak Leray-Hopf 
to the Navier-Stokes equations \eqref{e.nse} on $\R^3\times [t_1,t_1+S]$ with initial data $u(\cdot,t_1)\in L^2(\R^3)\cap L^p(\R^3)$.\\
Then, we have the following quantitative estimate
\begin{equation}\label{e.Lpenerest}
\mathcal E_{p,t_1}(S)\leq \|u(\cdot,t_1)\|_{L^p(\R^3)}^p+C_p\|u\|_{L^5(\R^3\times(t_1,t_1+S))}\mathcal E_{p,t_1}(S),
\end{equation}
where $\mathcal E_{p,t_1}(S)$ is defined by \eqref{e.Lpener}. Moreover, $C_p\stackrel{\infty}{\sim} p^\frac{3}{2}$.
\end{proposition}

This proposition will be used in the case $p=4$. We remark that if the critical $L^5(\R^3\times(t_1,t_1+T))$ is small enough, the $L^4$ estimate \eqref{e.Lpenerest} provides a control of $\mathcal E_{4,t_1}(T)$. This is the key idea behind the slicing argument in (Step 3) below.

\begin{proof}[Proof of Proposition \ref{prop.step2}]
We prove this result for the sake of completeness. The result may be well known, but we are not aware of an $L^p$ energy estimate in the form of \eqref{e.Lpenerest}; for a related estimate see for instance \cite[Theorem 0.1]{BDV87} and \cite[p. 45]{ESS2003}. Let us fix $p\geq 4$. 
 Estimate \eqref{e.Lpenerest} follows from testing the Navier-Stokes system \eqref{e.nse} with $u|u|^{p-2}$. This yields the following $L^p$ energy balance
\begin{align}
\label{e.enerbal}
&\frac{1}{p}\int\limits_{\R^3}|u(x,t)|^p\, dx+\int\limits_{t_1}^t\int\limits_{\R^3}|\nabla u|^2|u|^{p-2}\, dx+\frac{4(p-2)}{p^2}\int\limits_{t_1}^t\int\limits_{\R^3}|\nabla|u|^{\frac p2}|^2\, dx\\
&=\frac{1}{p}\int\limits_{\R^3}|u(x,t_1)|^p\, dx-\int\limits_{t_1}^t\int\limits_{\R^3}\nabla p\cdot|u|^{p-2}u\, dx+\int\limits_{t_1}^t\int\limits_{\R^3}(u\cdot\nabla u)\cdot|u|^{p-2}u\, dx\nonumber\\
&=\frac{1}{p}\int\limits_{\R^3}|u(x,t_1)|^p\, dx+I_{press}+I_{nl},\nonumber
\end{align}
for all $t\in[t_1,t_1+S]$. 
First, 
\begin{equation*}
I_{press}=(p-2)\int\limits_{t_1}^S\int\limits_{\R^3}R_{i}R_{j}(u_{i}u_{j})|u|^{p-4}u_{k}u_{l}\partial_{k}u_{l}\, dx,
\end{equation*}
where $\Riesz=(\Riesz_\alpha)_{\alpha=1,\ldots\, 3}$ is the Riesz transform. Here, we also utilized the Einstein summation convention. Using now H\"{o}lder's inequality and Calder\'on-Zygmund's theorem, we get the following estimate
\begin{align*}
|I_{press}|\leq\ & C(p-2)\||u|^2\|_{L^{\frac{5p}{p+3}}(\R^3\times(t_1,t))}\\
&\times\left(\int\limits_{t_1}^t\int\limits_{\R^3}|\nabla u|^2|u|^{p-2}\, dxds\right)^\frac12\||u|^\frac p2\|_{L^{\frac{10}3}(\R^3\times(t_1,t))}^\frac{p-2}{p}\\
\leq\ &C(p-2)\|u\|_{L^{5}(\R^3\times(t_1,t))}\||u|^\frac p2\|_{L^{\frac{10}3}(\R^3\times(t_1,t))}\left(\int\limits_{t_1}^t\int\limits_{\R^3}|\nabla u|^2|u|^{p-2}\, dxds\right)^\frac12
\end{align*}
where $C\in(0,\infty)$ is a universal constant. Notice that $C$ a priori depends on $p$ but can be chosen uniformly in $p\in[4,\infty)$. Indeed, 
it depends on the constant from Calder\'on-Zygmund's theorem for the $L^\frac{5p}{p+3}$ boundedness of the Riesz transforms, and $\frac{5p}{p+3}\rightarrow 5$ as $p\rightarrow\infty$. Second, by H\"{o}lder's inequality
\begin{align*}
|I_{nl}|\leq\ &\|u\|_{L^{5}(\R^3\times(t_1,t))}\||u|^\frac p2\|_{L^{\frac{10}3}(\R^3\times(t_1,t))}\left(\int\limits_{t_1}^t\int\limits_{\R^3}|\nabla u|^2|u|^{p-2}\, dxds\right)^\frac12.
\end{align*}
This implies \eqref{e.Lpenerest} by noticing that $\mathcal E_{p,t_1}(S)$ controls $\||u|^\frac p2\|_{L^{\frac{10}3}(\R^3\times(t_1,t))}$.
\end{proof} 

\noindent{\bf Step 3: transfer of subcritical information via slicing.} Let $M,\, E\in[1,\infty)$ be sufficiently large. Our goal is to prove that there exists $\delta(M,E)\in(0,\frac12]$ and $K(M,E)\in[1,\infty)$ such that for all $u_0\in L^2(\R^3)\cap L^4(\R^3)$ and any suitable weak Leray-Hopf solution associated to the initial data $u_0$, if 
\begin{equation*}
\|u_0\|_{L^2},\ \|u_0\|_{L^4}\leq M,
\end{equation*}
and
\begin{equation*}
\|u\|_{L^\infty(0,\infty;L^{3-\delta(M,E)}(\R^3))}\leq E,
\end{equation*}
then 
\begin{equation}\label{e.quantLinftyL4}
u\in L^{\infty}(0,T;L^{4}(\R^3))\Rightarrow\|u\|_{L^\infty(0,T;L^4(\R^3))}\leq K(M,E).
\end{equation}
This then obviously implies the result stated in Theorem \ref{theo.main}. The crucial point is that $K(M,E)$ is uniform in time. 
From the theory of mild solutions of the Navier-Stokes equations with subcritical initial data \cite{Giga86}, we have
$$\|u\|_{L^{\infty}(0, c_{4}\|u_0\|_{L^{4}}^{-8}; L^{4}(\mathbb{R}^3))}\leq 2\|u_{0}\|_{L^{4}}\leq 2M .$$ Hence it is sufficient to show that  
\begin{equation*}
{u\in L^{\infty}(0,T;L^{4}(\R^3))\Rightarrow}\|u\|_{L^\infty(c_{4}\|u_0\|_{L^{4}}^{-8},T;L^4(\R^3))}\leq K(M,E).
\end{equation*}
Note that $u$ is bounded and smooth on $\mathbb{R}^3\times [c_{4}\|u_0\|_{L^{4}}^{-8},T]$, hence all subsequent estimates can be justified rigorously.

The only a priori globally controlled quantity is a supercritical $L^\infty_tL^{3-}$ norm. We are not aware of any regularity mechanism enabling to brake the critically barrier based on the sole knowledge of such a supercritical bound. 
Therefore, the idea, following Bulut \cite{Bulut20} is to transfer the subcritical information coming from the initial data $u_0\in L^4(\R^3)$ to arbitrarily large times by using three ingredients:
\begin{enumerate}[label=(\arabic*)]
\item the control of the critical $L^\infty_tL^3_x$ norm via interpolation between the supercritical norm $L^\infty_tL^{3-\delta(M,E)}$ and the subcritical $L^\infty_tL^4_x$ norm;
\item the control of the critical $L^5_{t,x}$ norm from Proposition \ref{prop.step1}, which enables the slicing of the interval $(0,T)$ into a $T$-independent number of disjoint epochs $I_j=(t_j,t_j+S_j)$;
\item the $L^4$ energy estimate of Proposition \ref{prop.step2} which enables the transfer the subcritical information from time $t_j$ to $t_{j+1}$ and eventually to $T$.
\end{enumerate}
Let us carry out this three-step analysis, keeping the parameters $\delta:=\delta(M,E)\in(0,\frac12]$ and $K:=K(M,E)\in[1,\infty)$ free for the moment. These two parameters will be fixed at the end of the argument, see \eqref{e.defK} and \eqref{e.defdelta}. First by interpolation we obtain the following control
\begin{align}
\begin{split}\label{e.estLinftyL3}
\|u\|_{L^\infty(0,T;L^3(\R^3))}\leq\ &\|u\|_{L^\infty(0,T;L^{3-\delta}(\R^3))}^\frac{3-\delta}{3+3\delta}\|u\|_{L^\infty(0,T;L^4(\R^3))}^\frac{4\delta}{3+3\delta}\\
\leq\ &E^\frac{3-\delta}{3+3\delta}K^\frac{4\delta}{3+3\delta}.
\end{split}
\end{align}
Second, by the quantitative estimate \eqref{e.estL5tx} and the interpolation inequality \eqref{e.estLinftyL3}, we have
\begin{equation}\label{e.quantestinterpol}
\|u\|_{L^5(0,T;L^5(\R^3))}\leq C(M)\exp\exp\exp\big(C_{univ}\big(E^\frac{3-\delta}{3+3\delta}K^\frac{4\delta}{3+3\delta}\big)^c\big).
\end{equation}
We take $\ep:=\frac{1}{2C_4}$, where $C_4$ is the constant appearing in \eqref{e.Lpenerest} for $p=4$. We slice the interval $(c_{4}\|u_0\|_{L^{4}}^{-8},T)$ into $m+1$ disjoint successive intervals $I_j$ such that
\begin{equation}\label{e.cond1}
(c_{4}\|u_0\|_{L^{4}}^{-8},T)=\bigcup_{j\in\{1,\ldots\, m+1\}}I_j,\qquad I_j\cap I_{k}=\emptyset,\qquad I_j=(t_j,t_{j+1})=(t_j,t_j+S_j),
\end{equation}
where equality is up to a set of measure zero, and 
\begin{equation}\label{e.cond2}
\|u\|_{L^5(I_j;L^5(\R^3))}=\ep\quad\mbox{for all $1\leq j\leq m$ and}\quad \|u\|_{L^5(I_{m+1};L^5(\R^3))}\leq\ep.
\end{equation}
Notice that $t_1=c_{4}\|u_0\|_{L^{4}}^{-8}$ and $t_{m+2}=T$. 
We now remark that \eqref{e.quantestinterpol} provides an upper bound for the number of subintervals in the decomposition of $(0,T)$. Indeed, we have
\begin{align*}
\ep^5m\leq\sum_{j=1}^m\|u\|_{L^5(0,T;L^5(\R^3))}^5=\ &\|u\|_{L^5(0,T;L^5(\R^3))}^5\\
\leq\ &C(M)\exp\exp\exp\Big(\big(E^\frac{3-\delta}{3+3\delta}K^\frac{4\delta}{3+3\delta}\big)^c\Big).
\end{align*}
Hence, 
\begin{equation}\label{e.estm}
m\leq \frac{C(M)}{\ep^5}\exp\exp\exp\big(\big(E^\frac{3-\delta}{3+3\delta}K^\frac{4\delta}{3+3\delta}\big)^c\big).
\end{equation}
Finally, we estimate the growth of the $L^\infty_tL^4_x$ norm by iterating the $L^4$ energy estimate \eqref{e.Lpenerest} on the time intervals $I_j$. This yields
\begin{align*}
\mathcal E_{4,t_1}(S_1)\leq\ & \|u(\cdot,t_1)\|_{L^4(\R^3)}^4+C_4\|u\|_{L^5(\R^3\times I_1)}\mathcal E_{4,t_1}(S_1)\\
\leq\ &\|u(\cdot,t_1)\|_{L^4(\R^3)}^4+C_4\ep\mathcal E_{4,t_1}(S_1).
\end{align*}
Hence, by the definition of $\ep$ above, we can swallow the second term in the right hand side and get
\begin{equation*}
\mathcal E_{4,t_1}(S_1)\leq 2\|u(\cdot,t_1)\|_{L^4(\R^3)}^4\leq 32M^4.
\end{equation*}
Similarly, for all $1\leq j\leq m+1$, 
\begin{equation*}
\mathcal E_{4,t_j}(S_j)\leq 2\|u(\cdot,t_j)\|_{L^4(\R^3)}^4.
\end{equation*}
Iterating the above gives that for all $1\leq j\leq m+1$ 
\begin{equation*}
\mathcal E_{4,t_j}(S_j)\leq 2^j\|u(\cdot,t_1)\|_{L^4(\R^3)}^4\leq 64M^4 2^m.
\end{equation*}
Therefore, 
\begin{align*}
\|u\|_{L^\infty(c_{4}\|u_0\|_{L^{4}}^{-8},T;L^4(\R^3))}^4=\max_{1\leq j\leq m+1}\{\|u\|_{L^\infty(I_j;L^4(\R^3))}^4\}\leq 64M^4 2^m.
\end{align*}
Finally, we combine this last estimate with the estimate \eqref{e.estm} for $m$. We get
\begin{equation}\label{e.finalest}
\|u\|_{L^\infty(c_{4}\|u_0\|_{L^{4}}^{-8},T;L^4(\R^3))}^4
\leq 64M^4\exp\Big(\frac{2C(M)\log 2}{\ep^5}\exp\exp\exp\big(C_{univ}E^cK^\frac{4c\delta}{3+3\delta}\big)\Big),
\end{equation}
where we used $\frac{3-\delta}{3+3\delta}\leq 1$ for all $\delta\in(0,\frac12]$. For $M$ and $E$ sufficiently large, we choose now 
\begin{align}\label{e.defK}
K(M,E):=\ 3M\exp\Big(2^4(\log 2)C_4^5C(M)\exp\exp\exp\big(eC_{univ}E^c\big)\Big)
\in\ &[1,\infty)
\end{align}
and 
\begin{equation}\label{e.defdelta}
\delta(M,E):=\min\left(\frac{3}{4c\log(K(M,E))-3},\frac12\right)\in(0,\tfrac12].
\end{equation}
Notice that $\delta$ depends on $M$ and $E$ because $K$ does. With these choices, we have on the one hand,
\begin{equation*}
E^c(K(M,E))^\frac{4c\delta(M,E)}{3+3\delta(M,E)}\leq eE^c
\end{equation*}
and on the other hand
\begin{equation*}
64M^4\exp\Big(\frac{2C(M)\log 2}{\ep^5}\exp\exp\exp\big(E^cK^\frac{4c\delta}{3+3\delta}\big)\Big)\leq K(M,E)^4.
\end{equation*}
This concludes the proof of Theorem \ref{theo.main}.
\qed

Let us make some remarks on Theorem \ref{theo.main} and its proof, which will be particularly instructive for proving Theorem \ref{cor.orliczblowup}.

\begin{remark}\label{theo.mainlocalintime}
From the above proof of Theorem \ref{theo.main}, it is clear that if the hypothesis \eqref{slightlysupercrithypo} applies on $\mathbb{R}^3\times (0,S)$ $(0<S<\infty)$ instead of $\mathbb{R}^3\times (0,\infty)$, then the following holds true. Namely, 
$u\in L^{\infty}(0,S; L^{4}(\mathbb{R}^3))$ and hence 
 $u$ is smooth on $\mathbb{R}^3\times (0,S]$. Furthermore, notice that in view of the quantitative bound \eqref{e.quantLinftyL4}, suitability of $u$ is not required all the way up to time $S$.
\end{remark}
\begin{remark}\label{smallerdelta}
It is also clear from the above proof that replacing the hypothesis \eqref{slightlysupercrithypo} by
\begin{equation}\label{slightlysupercrithyposmallerdelta}
\|u\|_{L^{\infty}(0,\infty; L^{3-\hat{\delta}}(\mathbb{R}^3))}\leq E
\end{equation}
with $0<\hat{\delta}\leq \delta(M,E)$ gives the same conclusion as in Theorem \ref{theo.main}.

Consequently, if $M$ is sufficiently large and $E\geq M$, 
we see that the conclusion of Theorem \ref{theo.main} holds if the hypothesis \eqref{slightlysupercrithypo} is replaced by \eqref{slightlysupercrithyposmallerdelta} with  \begin{equation}\label{deltahatcondition}
 0<\hat{\delta}\leq\frac{1}{\exp\exp\exp(E^{2c})}.
 \end{equation} 
\end{remark}

\section{Blow-up of slightly supercritical Orlicz norms}
\label{sec.orlicz}

Let us first take $M$ sufficiently large as in the hypothesis of Theorem \ref{theo.main} and Remark \ref{smallerdelta}. We also choose  this $M$ such that
\begin{equation}\label{initidaldatabdM}
\|u_0\|_{L^2},\ \|u_0\|_{L^4}\leq M.
\end{equation}
To prove Theorem \ref{cor.orliczblowup}, it is sufficient to show that for $S\in (0,\infty)$, the assumption that $u$ is smooth on $\mathbb{R}^3\times (0,S)$ with 
\begin{equation}\label{e.orliczrecap}
\sup_{0<t<S}\int\limits_{\R^3}\frac{|u(x,t)|^3}{\Big(\log\log\log\big((\log(e^{e^{3e^{e}}}+|u(x,t)|))^\frac13\big)\Big)^\theta}dx\leq L<\infty
\end{equation}
and $L\geq 1$ implies that $u$ is smooth on $\mathbb{R}^3\times (0,S]$.

\noindent{\bf Step 1: estimating the $L^{3-\mu}_{x}$ norm of $u$.} Let $\mu\in (0,1)$ and $\lambda\in (1,\infty)$ be fixed parameters that will be determined below, see \eqref{mudefinition} and \eqref{lambdadef}. Note also that $\theta\in (0,1)$ is a universal constant that will be determined, see \eqref{thetadef}. For each fixed $t\in (0,S)$, we write
\begin{equation}\label{usplitting}
u(x,t)=u_{>\lambda}(x,t)+u_{\leq\lambda}(x,t),\,\,\,\textrm{where}\,\,\,u_{>\lambda}(x,t):=u(x,t)\chi_{\{x\in\mathbb{R}^3:|u(x,t)|>\lambda\}}.
\end{equation}
Our aim is to use \eqref{e.orliczrecap} to estimate the $L^{3-\mu}_{x}$ norm of $u_{>\lambda}(\cdot,t)$ and $u_{\leq\lambda}(\cdot,t)$ respectively. First, let us do this for $u_{>\lambda}(\cdot,t)$.

For each fixed $t\in (0,S)$ we perform the following splitting
\begin{equation}\label{Orliczsplit}
\int\limits_{\R^3}\frac{|u(x,t)|^3}{\Big(\log\log\log\big((\log(e^{e^{3e^{e}}}+|u(x,t)|))^\frac13\big)\Big)^\theta}dx=I_{>\lambda}(t)+I_{\leq\lambda}(t).
\end{equation}
Here,
\begin{equation}\label{Igreaterlambdadef}
I_{>\lambda}(t):=\int\limits_{\{x\in\mathbb{R}^3:|u(x,t)|>\lambda\}}\frac{|u(x,t)|^3}{\Big(\log\log\log\big((\log(e^{e^{3e^{e}}}+|u(x,t)|))^\frac13\big)\Big)^\theta}dx.
\end{equation}
There exists a  large constant $C(\theta)$ such that if
\begin{equation}\label{lambdafirstrestriction}
\lambda\geq \max(C(\theta),2)
\end{equation}
we have
$$\log(y)\geq \Big(\log\log\log\big((\log(e^{e^{3e^{e}}}+y))^\frac13\big)\Big)^\theta\,\,\,\,\,\,\forall y\in [\lambda,\infty). $$ Thus, if \eqref{lambdafirstrestriction} holds we have
\begin{equation}\label{firstinequalityIgeqlambda}
\int\limits_{\{x\in\mathbb{R}^3: |u(x,t)|>\lambda\}} |u(x,t)|^{3-\mu}\frac{|u(x,t)|^{\mu}}{\log(|u(x,t)|)} dx\leq I_{>\lambda}(t)\leq L.
\end{equation}
Let  $$f(y):=\frac{y^{\mu}}{\log(y)}.$$ Clearly
$$f'(y)\geq 0\,\,\,\textrm{for}\,\,\,y\geq \max(2, e^{\frac{1}{\mu}}).$$ Using this, together with \eqref{firstinequalityIgeqlambda}, we see that if
\begin{equation}\label{lambdasecondrestriction}
\lambda\geq\max(C(\theta),2,e^{\frac{1}{\mu}})
\end{equation} we have
\begin{equation}\label{L3-muugeqlambda}
\int\limits_{\{x\in\mathbb{R}^3: |u(x,t)|>\lambda\}} |u(x,t)|^{3-\mu} dx\leq \frac{\log(\lambda)}{\lambda^{\mu}}L.
\end{equation}
Next, let us estimate the $L^{3-\mu}_{x}$ norm of $u_{\leq\lambda}(\cdot,t)$. Clearly;
\begin{align}\label{uleqlambdaL3est}
\begin{split}
&\int\limits_{\{x\in\mathbb{R}^3:|u(x,t)|\leq\lambda\}}|u(x,t)|^3 dx\\
&=\int\limits_{\{x\in\mathbb{R}^3:|u(x,t)|\leq\lambda\}}\frac{|u(x,t)|^3\Big(\log\log\log\big((\log(e^{e^{3e^{e}}}+|u(x,t)|))^\frac13\big)\Big)^\theta}{\Big(\log\log\log\big((\log(e^{e^{3e^{e}}}+|u(x,t)|))^\frac13\big)\Big)^\theta}dx\\
&\leq L \Big(\log\log\log\big((\log(e^{e^{3e^{e}}}+\lambda))^\frac13\big)\Big)^\theta.
\end{split}
\end{align}
Now, since $u$ is a weak Leray-Hopf solution with initial data $\|u_{0}\|_{L^{2}(\mathbb{R}^3)}\leq M$, we have
\begin{equation}\label{uleqlambdaL2}
\int\limits_{\{x\in\mathbb{R}^3:|u(x,t)|\leq\lambda\}}|u(x,t)|^2 dx\leq M^2.
\end{equation}
From the interpolation of Lebesgue spaces we have
$$ \|u_{\leq\lambda}(\cdot,t)\|_{L^{3-\mu}(\mathbb{R}^3)}^{3-\mu}\leq \big(\|u_{\leq\lambda}(\cdot,t)\|_{L^{2}(\mathbb{R}^3)}^2\big)^{\mu}\big(\|u_{\leq\lambda}(\cdot,t)\|_{L^{3}(\mathbb{R}^3)}^3\big)^{1-\mu}.$$ 
Using this, together with \eqref{uleqlambdaL3est}-\eqref{uleqlambdaL2}, gives
\begin{equation}\label{uleqlambdaL3minusmuest}
\int\limits_{\{x\in\mathbb{R}^3:|u(x,t)|\leq\lambda\}}|u(x,t)|^{3-\mu} dx\leq M^{2\mu}\Big(L \Big(\log\log\log\big((\log(e^{e^{3e^{e}}}+\lambda))^\frac13\big)\Big)^\theta\Big)^{1-\mu}.
\end{equation}
By using \eqref{L3-muugeqlambda}, \eqref{uleqlambdaL3minusmuest} and noting that $M,L\geq 1$ and $\mu\in (0,1)$, we obtain the following. Namely, when $\lambda$ satisfies \eqref{lambdasecondrestriction} we have that for all $t\in (0,S)$
\begin{multline}\label{uL3minusmuest}
\|u(\cdot,t)\|_{L^{3-\mu}(\mathbb{R}^3)}\leq \max\Big(\Big(\frac{\log(\lambda)L}{\lambda^\mu}\Big)^{\frac{1}{2}},\Big(\frac{\log(\lambda)L}{\lambda^\mu}\Big)^{\frac{1}{3}}\Big)\\+ML^{\frac{1}{3}}\Big(\log\log\log\big((\log(e^{e^{3e^{e}}}+\lambda))^\frac13\big)\Big)^\frac{\theta}{3}.
\end{multline} 
Note that since $L\geq 1$, this bound is larger than $M$.
\\
\noindent{\bf Step 2: choosing $(\mu,\lambda,\theta)$ and applying Theorem \ref{theo.main}.}
Define 
\begin{equation}\label{thetadef}
\theta:=\frac{1}{2c}\in (0,1),
\end{equation}
where $c\in(1,\infty)$ is the universal constant appearing in Proposition \ref{prop.step1}, equation \eqref{e.estL5tx}. Next define
\begin{equation}\label{mudefinition}
\mu:=\frac{1}{(\log(\lambda))^{\frac{1}{2}}},
\end{equation}
where $\lambda\geq 2$ is to be determined. For $\lambda\geq e$ it is clear that $\lambda\geq e^{\frac{1}{\mu}}$. Furthermore, $$\lim_{\lambda\uparrow\infty}\Big(\frac{\log(\lambda)}{\lambda^{\mu}}\Big)=0. $$
Hence, there exists a $\lambda_{0}(M,L)$ such that for all $\lambda\geq \lambda_{0}(M,L)$ we have
\begin{equation}\label{logbounds}
\Big(\log(e^{e^{3(e^e)}}+\lambda)\Big)^{\frac{1}{3}}\leq (\log(\lambda))^{\frac{1}{2}}
\end{equation}
and
\begin{equation}\label{L3minusmuboundE}
\max\Big(\Big(\frac{\log(\lambda)L}{\lambda^\mu}\Big)^{\frac{1}{2}},\Big(\frac{\log(\lambda)L}{\lambda^\mu}\Big)^{\frac{1}{3}}\Big)+ML^{\frac{1}{3}}\Big(\log\log\log\big((\log(e^{e^{3e^{e}}}+\lambda))^\frac13\big)\Big)^\frac{1}{6c}\leq E. 
\end{equation}
 Here,
\begin{equation}\label{Edef}
E:=\Big(\log\log\log\big((\log(e^{e^{3e^{e}}}+\lambda))^\frac13\big)\Big)^\frac{1}{2c}.
\end{equation}
Now we choose any 
\begin{equation}\label{lambdadef}
\lambda\geq\max\Big(C(\theta),e,\lambda_{0}(M,L)\Big),
\end{equation}
with $C(\theta)$ being as in \eqref{lambdasecondrestriction}. 
With this choice of $\lambda$, we see that the conclusion \eqref{uL3minusmuest} holds. Hence, we obtain from \eqref{L3minusmuboundE} that
\begin{equation}\label{uL3minusmusecondest}
\|u\|_{L^{\infty}(0,S; L^{3-\mu}(\mathbb{R}^3))}\leq E\,\,\,\textrm{with}\,\,\,E\geq M.
\end{equation}
From \eqref{logbounds}, we clearly have
\begin{equation}\label{muupper}
0<\mu\leq \frac{1}{\exp\exp\exp(E^{2c})}.
\end{equation}
Taking $\hat{\delta}:=\mu$ and observing Remark \ref{theo.mainlocalintime} and Remark \ref{smallerdelta}, we see that we can apply Theorem \ref{theo.main} to obtain that $u$ is smooth on $\mathbb{R}^3\times (0,S]$. This concludes the proof.\qed

\section{Final discussion: the case of Type I bounds}
At first sight one might wonder if Theorem \ref{theo.main} could be used to resolve the following longstanding open problem. Namely, whether a suitable weak Leray-Hopf solution, with initial data $u_{0}$, satisfying the conditions
\begin{equation}\label{initialdataassumption}
\|u_{0}\|_{L^2},\|u_{0}\|_{L^{4}}\leq M
\end{equation}
and 
\begin{equation}\label{TypeIbound}
\|u\|_{L^{\infty}(0,\infty; L^{3,\infty}(\mathbb{R}^3))}\leq M
\end{equation}
 (for arbitrarily large $M$) is smooth on $\mathbb{R}^3\times (0,\infty)$. In the literature \eqref{TypeIbound} is sometimes known as a `Type I bound'.
 
Unfortunately, it seems that Theorem \ref{theo.main} cannot be used to resolve this open problem. Indeed, using interpolation of Lorentz spaces \cite[Lemma 2.2]{mccormick2013generalised} we see that for $\delta\in (0,\frac{1}{2}]$ and 
 $$\theta=\frac{6}{3-\delta}-2 $$
 we have the following:
 \begin{equation}\label{interpolationweakL3}
 \|u(\cdot,t)\|_{L^{3-\delta}}\leq \Big(\frac{3-\delta}{1-\delta}+\frac{3-\delta}{\delta}\Big)\|u(\cdot,t)\|_{L^{2,\infty}}^{\theta}\|u(\cdot,t)\|_{L^{3,\infty}}^{1-\theta}\leq \frac{6M}{\delta}.
 \end{equation}
 Hence, $$\|u\|_{L^{\infty}(0,\infty; L^{3-\delta}(\mathbb{R}^3))}\leq E(M,\delta):= \frac{6M}{\delta}. $$
 For $E$ defined in this way, the requirement \eqref{e.defdelta} (coming from the proof of Theorem \ref{theo.main}) becomes an equation for $\delta$ for which there is no solution when $M$ is large. Hence \eqref{e.defdelta} cannot be satisfied.

\subsection*{Funding and conflict of interest.} 

The first author is supported by a Leverhulme Early Career Fellowship funded by The Leverhulme Trust. The second author is partially supported by the project BORDS grant ANR-16-CE40-0027-01 and by the project SingFlows grant ANR-18-CE40-0027 of the French National Research Agency (ANR). The authors declare that they have no conflict of interest. 
\subsection*{Acknowledgement}
The second author is grateful to Patrick G\'erard for bringing to his attention the result of Aynur Bulut \cite{Bulut20}.

\small 
\bibliographystyle{abbrv}
\bibliography{concentration.bib}

\end{document}